\documentclass[preprint,12pt]{amsart}
\usepackage{tikz-cd}
\usepackage{amssymb}
\usepackage{amsthm}
\usepackage{amsmath}
\usepackage{mathrsfs}

\newtheorem{theorem}{Theorem}[section]
\newtheorem{lemma}{Lemma}[section] 
\newtheorem{proposition}{Proposition}[section] 
\newtheorem{corollary}{Corollary}[section] 
\newtheorem*{corollary*}{Corollary}
\newtheorem{theoremA}{Theorem}

\theoremstyle{definition} 
\newtheorem{definition}{Definition}[section]
\newtheorem*{definition*}{Definition}
\newtheorem{question}{Question}
\newtheorem*{question-4-prime}{Question~4$\boldsymbol{'}$}

\theoremstyle{remark} 
\newtheorem{remark}{Remark}[section]

\def\bigdiag{\operatornamewithlimits{%
  \mathchoice{\vcenter{\hbox{\Large$\Delta$}}}
             {\vcenter{\hbox{\large$\Delta$}}}
             {\Delta}
             {\Delta}}}
\def\diag{\operatornamewithlimits{%
  \mathchoice{\vcenter{\hbox{$\Delta$}}}
             {\vcenter{\hbox{$\Delta$}}}
             {\Delta}
             {\Delta}}}

\begin{document}

{
\renewcommand*{\thefootnote}{$\star$}

\title{Weird $\mathbb R$-Factorizable Groups}
\footnotetext[0]{The work of O.~Sipacheva was financially supported by the Russian Science Foundation, 
grant~22-11-00075-P.} 

}

\author{Evgenii Reznichenko}

\author{Ol'ga Sipacheva}

\email[E. Reznichenko]{erezn@inbox.ru} 

\email[O. Sipacheva]{osipa@gmail.com}

\address{Department of General Topology and Geometry, Faculty of Mechanics and  Mathematics, 
M.~V.~Lomonosov Moscow State University, Leninskie Gory 1, Moscow, 199991 Russia}

\begin{abstract}
The problem of the existence of non-pseudo-$\aleph_1$-com\-pact $\mathbb R$-factorizable groups is studied. 
It is proved that any such group is submetrizable and has weight larger than $\omega_1$. Closely related 
results concerning the $\mathbb R$-factorizability of products of topological 
groups and spaces are also obtained (a product $X\times Y$ of topological spaces is said to be  
$\mathbb R$-factorizable if any continuous function 
$X\times Y\to \mathbb R$ factors through a product of maps from $X$ and $Y$ to second-countable spaces). In 
particular, it is proved that the square $G\times G$ of a topological groups $G$ 
is $\mathbb R$-factorizable as a group  if and only if it is $\mathbb R$-factorizable as a 
product of spaces, in which case $G$ is pseudo-$\aleph_1$-compact. It is also proved that if the product 
of a space $X$ and an uncountable discrete space is $\mathbb R$-factorizable, then $X^\omega$ is heredirarily 
separable and heredirarily Lindel\"of. 
\end{abstract}



\maketitle

In the middle of the past century Pontryagin proved that any continuous function on a compact topological 
group factors through a continuous homomorphism to a second-countable group (see, e.g., 
\cite[Example~37]{Pontryagin}). This result gave rise to the theory of $\mathbb R$-factorizable groups, which 
has been fruitfully developed since then. 

\begin{definition*}[\cite{Tk1}]
A topological group $G$ is said to be \emph{$\mathbb R$-factorizable} if any continuous function 
$f\colon G\to \mathbb R$ \emph{factors through} a homomorphism to a second-countable group, i.e., there exists a 
second-countable topological group $H$, a continuous homomorphism $h\colon G \to H$, and a continuous function 
$g\colon H\to \mathbb R$ for which $f = g \circ h$. 
\end{definition*}

The notion of an $\mathbb R$-factorizable group was explicitly introduced by Tka\-chen\-ko \cite{Tk1}, 
who also obtained the first fundamental results. Among other things he proved that 
the following topological groups are $\mathbb R$-factorizable:
\begin{itemize}
\item
any Lindel\"of group;
\item
any totally bounded group, that is, a group $G$ such that, given any open neighborhood $U$ of its 
identity element, there exists a finite subset $A$ of $G$ for which $AU = G$ (totally bounded groups are 
precisely subgroups of compact groups); 
\item
any subgroup of a Lindel\"of $\Sigma$-group, in particular, any subgroup of a $\sigma$-compact group;
\item 
any dense subgroup of an arbitrary product of Lindel\"of $\Sigma$-groups.
\end{itemize}

In the decades since Tkachenko's paper \cite{Tk1} was published, the theory of $\mathbb R$-factorizable groups has been 
extensively developed; it is surveyed in Chapter~8 of the book~\cite{AT}. 

However, many problems concerning $\mathbb R$-factorizable groups remain open. We consider the four questions 
posed below (see also \cite{Tk1}) to be the most important of them. Recall that a topological space $X$ 
is said to be \emph{pseudo-$\aleph_1$-compact} if any 
locally finite (or, equivalently, any discrete) family  of open sets in $X$ is at most countable. 

\begin{question}
\label{q1}
Is any $\mathbb R$-factorizable group pseudo-$\aleph_1$-compact?
\end{question}

\begin{question}
\label{q2}
Is the image of an $\mathbb R$-factorizable group under a continuous homomorphism $\mathbb R$-factorizable?
\end{question}

Note that Question~\ref{q2} is equivalent to the question of whether the image of an $\mathbb R$-factorizable 
group under a continuous isomorphism $\mathbb R$-factorizable, because the quotients  
of $\mathbb R$-factorizable groups are $\mathbb R$-factorizable~\cite[Theorem 8.4.2]{AT}.

\begin{question}
\label{q3}
Is the square of an $\mathbb R$-factorizable group $\mathbb R$-factorizable?
\end{question}

\begin{question}
\label{q4}
Is the property of being $\mathbb R$-factorizable topological in the class of topological groups? 
In other words, is any topological group homeomorphic to an $\mathbb R$-factorizable one 
$\mathbb R$-factorizable? 
\end{question}

If $H$ is a topological group and $D$ is a discrete uncountable topological group, then 
the group  $H\times D$ is not $\mathbb R$-factorizable, because it is not $\omega$-narrow. Thus, the following 
question is of interest in relation to Question~\ref{q4}.

\begin{question-4-prime}
Is it true that no $\mathbb R$-factorizable group is homeomorphic to a product 
$H\times D$, where $H$ is a topological group and $D$ is an uncountable discrete space? 
\end{question-4-prime}

The first question is most intriguing, at least because if the answer to it is negative, then so are the answers 
to Questions~\ref{q2} and~\ref{q3}. We prove in this paper that the above questions are related as shown in the 
following diagram. An arrow $A\to B$ means that if the answer to Question~$A$ is positive, then so is the answer 
to Question~$B$.  
\[ 
\begin{tikzcd} 2 \arrow[r] & 1 \arrow[d] & 3 \arrow[l] \arrow[d] \\ 
& 4' & 4 \arrow[l] 
\end{tikzcd} 
\]

According to Theorem 8.5.2 of \cite{AT}, a topological group $G$ is $\mathbb R$-factorizable and 
pseudo-$\aleph_1$-compact if and only if it is \emph{$m$-factorizable}, that is, 
any  continuous map $f\colon G\to M$ to any metrizable space $M$ factors through a continuous 
homomorphism to a second-countable topological group. A class of $m$-factorizable groups is very important; 
see Section~8.5 of \cite{AT}. Question~\ref{q1} can be formulated as follows: Is any $\mathbb R$-factorizable 
group $m$-factorizable?

As the question of the existence of non-pseudo-$\aleph_1$-compact $\mathbb R$-factorizable groups is so important, 
we give these groups a name. 

\begin{definition*}
An $\mathbb R$-factorizable group which is not pseudo-$\aleph_1$-compact is called a \emph{weird 
$\mathbb R$-factorizable} group. 
\end{definition*}

Yet another way to state Question~\ref{q1} is: Is it true that weird $\mathbb R$-factorizable 
groups do not exist?

Weird $\mathbb R$-factorizable groups have rather abnormal properties. In this paper we obtain results which 
imply the following theorem. 

\begin{theoremA}
\label{tA}
Let $G$ be a weird $\mathbb R$-factorizable group. Then
\begin{enumerate}
\item[\rm(1)]
$G\times G$ is not $\mathbb R$-factorizable;
\item[\rm(2)]
there exists a surjective continuous homomorphism of $G$ to a non-$\mathbb R$-factorizable group;
\item[\rm(3)]
$w(G)>\omega_1$;
\item[\rm(4)]
$w(G)^\omega\geq 2^{\omega_1}$;
\item[\rm(5)]
$\psi(G)\leq\omega$, that is, $G$ is submetrizable;
\item[\rm(6)]
if $H$ is a topological group and the group $G\times H$ is $\mathbb R$-factorizable, then 
\begin{enumerate}
\item[\rm(a)]
$H^\omega$ is hereditarily Lindel\"of and hereditarily separable; 
\item[\rm(b)]
if $w(H)\leq \omega_1$, then $H$ is second-countable;
\item[\rm(c)]
under $\mathrm{CH}$, 
$H$ is second-countable. 
\end{enumerate}
\end{enumerate}
\end{theoremA}

Assertion~(1) follows from Exercise~8.5.a in~\cite{AT} (see also Corollary~\ref{c4.1} in the present paper), 
and assertion~(4) follows from Theorem~8.5.8 in~\cite{AT}. The other assertions of the theorem are new; 
their proofs are given in Section~\ref{s4}.

It is seen from Theorem~\ref{tA} that the problem of the existence of weird $\mathbb R$-factorizable groups is 
closely related to the $\mathbb R$-factorizability of products of groups. 
Only recently has the fundamental question of the multiplicativity of the class of $\mathbb R$-factorizable 
groups been answered by constructing Lindel\"of (and hence $\mathbb R$-factorizable) groups $G$ and $H$ whose product 
$G\times H$ is not $\mathbb R$-factorizable \cite{Reznichenko2024rf,Sipacheva2023rf}\footnote{As shown in 
\cite{Reznichenko2024rf}, the products of Lindel\"of groups constructed in \cite{Sipacheva2023cd,Sipacheva2022pg} 
are not $\mathbb R$-factorizable either.}; moreover, one of these groups can be made second-countable. Section~\ref{s3} of the present paper is devoted to 
$\mathbb R$-factorizable products of groups. 

The key role in the study of $\mathbb R$-factorizable products of groups is played by the notion of 
an \emph{$\mathbb R$-factorizable product of spaces}.

\begin{definition*}[\cite{ReznichenkoSipacheva2013,Reznichenko2024rf}]
Given topological spaces $X_1, \dots, X_n$ and $Y$, we say that a map $f\colon X_1\times\dots\times X_n\to Y$ 
is \emph{$\mathbb R$-factorizable} if it \emph{factors through a product} of continuous maps to second-countable 
spaces, that is, if there exist 
second-countable spaces $X'_1,\dots, X'_n$ and continuous maps $g_i\colon X_i\to X'_i$, $i\le n$,
and $h\colon X'_1\times \dots \times X'_n\to Y$ such that $f=h\circ(g_1\times \dots \times g_n)$, i.e., the following diagram 
is commutative: 
\[ 
\begin{tikzcd} X_1\times\dots\times X_n \arrow[rr, "f"]  \arrow[rd, "g_1\times\dots\times g_n"']   && Y \\ 
& X'_1\times\dots\times X'_n \arrow[ru, "h"'] 
\end{tikzcd} 
\] 
We say that a product $X_1\times\dots\times X_n$ is \emph{$\mathbb R$-factorizable} (or 
\emph{multiplicatively $\mathbb R$-factorizable}, when there is a danger of confusion) 
if any continuous function $f\colon X_1\times\dots\times X_1 \to \mathbb R$ is $\mathbb R$-factorizable. 
\end{definition*} 

The notion of an $\mathbb R$-factorizable map was introduced in \cite{ReznichenkoSipacheva2013}, where 
it was essentially proved that, for the free topological group $F(S)$ of the Sorgenfrey line $S$, 
the product $F(S)\times F(S)$ is not multiplicatively $\mathbb R$-factorizable. In turn, the notion of a 
multiplicatively $\mathbb R$-factorizable product was introduced in \cite{Reznichenko2024rf}; in the same paper, the 
following statement was proved, which is the main tool for constructing topological groups whose products are not 
$\mathbb R$-factorizable (as groups). 

\begin{theoremA}[{\cite[Corollary~2(2)]{Reznichenko2024rf}}] 
\label{tB}
If the product $G\times H$ of topological groups $G$ and $H$ is an $\mathbb R$-factorizable group, then 
$G\times H$ is a multiplicatively $\mathbb R$-factorizable product.
\end{theoremA} 

In this paper, we refine Theorem~\ref{tB} as follows (see Theorem~\ref{t3.1} in Section~\ref{s3}). 

\begin{theoremA}
\label{tC}
For topological groups $G$ and $H$, the following conditions are equivalent:
\begin{enumerate}
\item[\rm (1)]
the group $G\times H$ is $\mathbb R$-factorizable;
\item[\rm (2)]
$G$ and $H$ are $\mathbb R$-factorizable and the product $G\times H$ is multiplicatively $\mathbb R$-factorizable. 
\end{enumerate} 
\end{theoremA}

In \cite{EAR} the notion of $\mathbb R$-factorizability was extended from the class of topological 
groups to the much larger class of topological universal algebras. Results of \cite{EAR} 
have enabled us to prove the following statement (see Theorem~\ref{t3.2} below).

\begin{theoremA}
\label{tD}
Let $G$ be a topological group. Then the group $G\times G$ is $\mathbb R$-factorizable if and only if 
the product $G\times G$ is multiplicatively $\mathbb R$-factorizable. 
\end{theoremA}

\begin{corollary*}
Given a topological group $G$, the $\mathbb R$-factorizability of the group $G\times G$ is a topological 
property of $G$. In other words, if the group $G\times G$ is $\mathbb R$-factorizable and  
$H$ is a topological group homeomorphic to $G$, then the group $H\times H$ is $\mathbb R$-factorizable. 
\end{corollary*}

This corollary shows why a positive answer to Question~\ref{q3} gives a positive answer to Question~\ref{q4}. 

In \cite{Blair-Hager} Blair and Hager considered conditions under which a product $X\times Y$ is 
$z$-embedded in $\beta X\times \beta Y$. We will show in Section~\ref{s2}, which is devoted to 
multiplicatively $\mathbb R$-factorizable products, that a product $X\times Y$ has this property if and only if it is 
$\mathbb R$-factorizable (see Proposition~\ref{p2.1}). In the same section we also prove the following 
statement (this is Theorem~\ref{t2.1}).

\begin{theoremA}
\label{tE}
Suppose that a product $X\times Y$ is $\mathbb R$-factorizable and $Y$ is not 
pseudo-$\aleph_1$-compact. Then 
\begin{enumerate}
\item[\rm(1)]
$X^\omega$ is hereditarily Lindel\"of and hereditarily separable; 
\item[\rm(2)] if $w(X)\leq \omega_1$, then $X$ is second-countable; 
\item[\rm(3)] under \textup{CH}, $X$ is second-countable. 
\end{enumerate} 
\end{theoremA} 

This theorem strengthens Proposition~2.1(b) of~\cite{Blair-Hager}. 

The CH assumption cannot be omitted from (3), because it has recently been shown by the authors 
jointly with Anton Lipin that if $\mathfrak b>\omega_1$, then the product of the countable 
Fr\'echet--Urysohn fan and a discrete space of cardinality $\omega_1$ is $\mathbb R$-factorizable. 

\section{Preliminaries} 

All topological spaces and groups considered in this paper are assumed to be Tychonoff. 
Throughout the paper, by a space we mean a topological space and use $I$ to denote an arbitrary index set. Given 
ordinals $\alpha$ and $\beta$, by $[\alpha,\beta]$ we denote the set of all ordinals $\gamma$ satisfying 
the inequalities $\alpha\le \gamma\le \beta$. The weight of a space $X$ is denoted by~$w(X)$ and its 
pseudocharacter (that is, the least cardinal $\kappa$ such that every point of $X$ is the intersection of 
at most $\kappa$-many neighborhoods), by $\psi(X)$. By a \emph{Lindel\"of $\Sigma$-group} we mean a topological 
group whose underlying topological space is a Lindel\"of $\Sigma$-space, that is, can be represented as a 
continuous image of a perfect preimage of a second-countable space. A topological space $X$ is 
said to be \emph{perfectly $\varkappa$-normal} if the closure of any open set in $X$ is a zero set. 

\begin{definition}
Given a cardinal $\kappa$, a topological space $X$ is said to be \emph{pseudo-$\kappa$-compact} if the 
cardinality of any locally finite family of open sets in $X$ is less than~$\kappa$. 
\end{definition}

\begin{remark} 
It is easy to see that a Tychonoff space $X$ is pseudo-$\kappa$-compact if and only if the cardinality of 
any discrete family of cozero sets in $X$ is less than~$\kappa$. 

Indeed, suppose that there exists a locally finite family $\{U_\alpha: \alpha\in \kappa\}$ 
of nonempty open sets in $X$. Let us show that there exists a discrete family of cozero sets which has 
cardinality $\kappa$. Choose points $x_\alpha\in X$ and cozero sets $V_\alpha\subset X$ so that $x_\alpha\in 
V_\alpha\subset U_\alpha$ and $V_\alpha$ intersects only finitely many sets $U_\beta$ for each $\alpha< \kappa$. 
Let $f\colon \kappa\to \kappa$ be a function such that $V_\alpha\cap V_\beta=\varnothing$ for all $\beta\ge 
f(\alpha)$. Clearly, $f$ is increasing and the family $\{V_{f(\alpha)}:\alpha\in \kappa\}$ is a locally finite 
disjoint (and hence discrete) family of cozero sets. 
\end{remark}

\begin{definition}
A topological group $G$ is said to be \emph{$\omega$-narrow} if, for every open neighborhood $V$ of the
identity element in $G$, there exists a countable set $A \subset G$ such that $V\!A = G$ (or, equivalently, such 
that $AV=G$). 
\end{definition}

According to \cite[Proposition~8.1.3]{AT}, every $\mathbb R$-factorizable group is $\omega$-narrow.

\begin{definition}
A subspace $Y$ of a topological space $X$ is said to be \emph{$z$-embedded} in $X$ if, for every zero set $Z$ 
in $Y$, there exists a zero set $F$ in $X$ such that $F \cap Y = Z$. 
\end{definition}

Clearly, ``zero'' in this definition can be replaced by ``cozero.'' It is also clear that any $z$-embedded 
subspace $Z$ of a $z$-embedded subspace $Y$ of a space $X$ is $z$-embedded in~$X$. 

The property of being $z$-embedded plays a crucial role in the theory of $\mathbb R$-factorizable groups, because 
\emph{an $\omega$-narrow topological group $G$ is $\mathbb R$-factorizable if and only if $G$ is $z$-embedded in every 
topological group that contains $G$ as a topological subgroup}~\cite[Theorem~8.2.7]{AT}.

Recall that a subspace $Y$ of a space $X$ is \emph{$C$-embedded} ($C^*$-embedded) in $X$ if 
any continuous (any bounded continuous) function $f\colon Y\to \mathbb R$ has a continuous extension to $X$. 
Obviously, all $C$- and $C^*$-embedded subspaces of $X$ are $z$-embedded. It is also known that a zero set 
$Y$ of $X$ is $z$-embedded in $X$ if and only if $Y$ is $C^*$-embedded in $X$ \cite[Corollary~11.7]{Ch}. 
On the other hand, any cozero set of $X$ is $z$-embedded in $X$ \cite[Lemma~11.12]{Ch}. 

\begin{remark}
The union of any discrete family of cozero sets in a space $X$ is $z$-embedded in~$X$. 

Indeed, let 
 $\mathscr U=\{U_\iota:\iota \in I\}$ be a discrete family of cozero sets, and let $f_\iota\colon X\to [0,1]$ 
witness their being cozero. The function 
$$
f=\sum_{\iota\in I}f_\iota\colon X\to \mathbb R
$$ 
is well defined (because 
$\mathscr U$ is disjoint) and continuous (because $\mathscr U$ is discrete and hence locally finite). Clearly, 
$f^{-1}(\{0\}) = X\setminus \bigcup \mathscr U$. Thus, $\bigcup \mathscr U$ is a cozero set in $X$; therefore, it is 
$z$-embedded.  
\end{remark}

\begin{remark}
\label{r1.3}
Let $\mathscr U=\{U_\iota:\iota \in I\}$ be a discrete family of nonempty cozero sets in a space $X$, and let 
$y_\iota\in U_\iota$ for each $\iota\in I$. Then the (discrete) set $Y=\{y_\iota:\iota\in I\}$ is $C$-embedded in 
$X$.

Indeed, let $f_\iota\colon X\to \mathbb R$ be continuous functions witnessing that the $U_\iota$ are cozero, and let 
$g\colon Y\to \mathbb R$ be any function. The function 
$$
f=\sum_{\iota\in I}\frac{g(y_\iota)}{f_\iota(y_\iota)}\cdot f_\iota\colon X\to \mathbb R
$$ 
is well defined and continuous, and its restriction to $Y$ coincides with~$g$.
\end{remark}

Recall that a topological space is \emph{submetrizable} if it admits a coarser metrizable topology. In this paper 
we repeatedly use the following well-known theorems. 

\begin{theorem}[{see, e.g., \cite[Theorem~3.3.16]{AT}}]
\label{t1.1}
A topological group is submetrizable if and only if it has countable pseudocharacter.
\end{theorem}

\begin{theorem}[{\cite[Theorem~8.4.2]{AT}}]
\label{t1.2}
The image of an $\mathbb R$-factorizable group under a quotient \textup($=$\,open\textup) homomorphism is 
$\mathbb R$-factorizable. \end{theorem}

\section{$\mathbb R$-Factorizable Products} 
\label{s2}

\begin{definition}
A \emph{cozero rectangle} in a product $X\times Y$ of topological spaces $X$ and $Y$ is any set of the form 
$V\times W$, where $V$ and $W$ are cozero sets in $X$ and $Y$, respectively. 
\end{definition} 

\begin{proposition}
\label{p2.1}
For any spaces $X$ and $Y$, the following conditions are equivalent: 
\begin{enumerate}
\item[\rm (1)]
$X\times Y$ is $\mathbb R$-factorizable;
\item[\rm (2)]
any cozero set in $X\times Y$ is a countable union of cozero rectangles;
\item[\rm (3)]
$X\times Y$ is $z$-embedded in $\beta X\times \beta Y$;
\item[\rm (4)]
if $X$ is $z$-embedded in $X'$ and $Y$ is $z$-embedded in $Y'$, then $X\times Y$ is $z$-embedded in $X'\times Y'$;
\item[\rm (5)]
there exist spaces $X'$ and $Y'$ such that $X'\times Y'$ is $\mathbb R$-factorizable and 
$X\times Y$ is $z$-embedded in $X'\times Y'$; 
\item[\rm(6)] there exist spaces $X'$ and $Y'$ such that 
$X'\times Y'$ is Lindel\"of and $X\times Y$ is $z$-embedded in $X'\times Y'$. 
\end{enumerate} 
\end{proposition} 

\begin{proof} First, we prove the implication $(1)\Rightarrow(2)$. Let $U$ be a cozero set in $X \times 
Y$, and let $f$ be a continuous function on $X \times Y$ such that $(X \times Y)\setminus U = f^{-1}(\{0\})$. 
Since $X \times Y$ is $\mathbb R$-factorizable, it follows that there exist second-countable spaces $M$ and $H$, 
continuous maps $p\colon X \to M$  and $q\colon Y \to H$, and a continuous function $h\colon M\times H\to \mathbb R$ 
such that $f = h\circ (p \times q)$. The preimage $h^{-1}(\mathbb R\setminus\{0\})$ is a cozero set in the 
second-countable Tychonoff space $M\times H$; hence there exist cozero sets $V_i\subset M$ and $W_i\subset H$, 
$i\in \omega$, such that $h^{-1}(\mathbb R\setminus\{0\})=\bigcup_{i\in \omega}(V_i\times W_i)$. 
We have $U=\bigcup_{i\in \omega}(p^{-1}(V_i)\times q^{-1}(W_i))$, and the $p^{-1}(V_i)$ and $q^{-1}(W_i)$ are cozero sets in $X$ 
and $Y$, respectively. 

Now we show that $(2)\Rightarrow(1)$. 
Let $\mathscr B=\{U_n:n\in \omega\}$ be a countable base of the topology of $\mathbb R$, and let $f\colon X\times Y\to 
\mathbb R$ be a continuous function.  By assumption, for each $n\in \omega$, 
$$
f^{-1}(U_n)= \bigcup_{i\in \omega}(V_{n,i}\times W_{n,i}),
$$ 
where the $V_{n,i}$ and $W_{n,i}$ are cozero sets in $X$ and $Y$, respectively. Let $g_{n,i}\colon X\to \mathbb R$ and 
$h_{n,i}\colon Y\to \mathbb R$ be continuous functions for which $V_{n,i}=g_{n,i}^{-1}(\mathbb R\setminus \{0\})$ and 
$W_{n,i}=h_{n,i}^{-1}(\mathbb R\setminus \{0\})$. We set 
$$
g=\bigdiag_{n,i\in \omega} g_{n,i}\colon X\to \mathbb R^{\omega\times \omega}\quad\text{and}\quad 
h=\bigdiag_{n,i\in \omega} h_{n,i}\colon Y\to \mathbb R^{\omega\times \omega}.
$$
We have 
$$
(g\times h)^{-1}\bigl((g\times h)(f^{-1}(U_n))\bigr) = f^{-1}(U_n),
$$
because $(x,y)\notin f^{-1}(U_n)$ if and only if 
$$
(g\times h)\bigl((x,y)\bigr) = \bigl((a_{n,i})_{(n,i)\in 
\omega\times \omega},(b_{n,i})_{(n,i)\in \omega\times \omega}\bigr),
$$ 
where $a_{n,i}\cdot b_{n,i}=0$ for all $i\in \omega$. Note that if $f\bigl((x,y)\bigr)\ne 
f\bigl((x',y')\bigr)$, then there exists an $n\in \omega$ for which $f\bigl((x,y)\bigr)\in U_n$ and 
$f\bigl((x',y')\bigr)\notin U_n$, whence $(g\times h)\bigl((x,y)\bigr)\ne (g\times h)\bigl((x',y')\bigr)$. 
Therefore, setting 
$$
\varphi\bigl((a,b)\bigr) = 
f(x,y), \quad \text{where\quad $(x,y)$ is any point in $(g\times 
h)^{-1}\bigl(\{(a,b)\}\bigr)$,}
$$
we obtain a well-defined function $\varphi\colon (g\times h)(X\times Y)\to \mathbb R$. 
For each $n\in \omega$, 
\begin{multline*}
\varphi^{-1}(U_n)=\\
\bigcup_{i\in \omega}\bigl\{\bigl((a_{n,i})_{(n,i)\in 
\omega\times \omega},(b_{n,i})_{(n,i)\in \omega\times \omega}\bigr)\in (g\times h)(X\times Y): a_{n,i}\cdot 
b_{n,i}\ne 0 \bigr\}.
\end{multline*} 
All these sets are open in $(g\times h)(X\times Y)$ and hence $\varphi$ is continuous. Clearly, 
$f=\varphi\circ (g\times h)$. 

The equivalences $(2)\Leftrightarrow(3)\Leftrightarrow (4)$ follow from Theorem~1.1 of \cite{Blair-Hager}.

Let us prove $(5)\Leftrightarrow(6)$. Since any cozero set is $F_\sigma$, Lindel\"ofness is preserved by $F_\sigma$ subspaces, and any open set in 
a product of Tychonoff spaces is a union of cozero rectangles, 
it follows from $(1)\Leftrightarrow (2)$ that all Lindel\"of products are 
$\mathbb R$-factorizable. Therefore, $(6)\Rightarrow(5)$. 
To see that $(5)\Rightarrow(6)$, it suffices to note that if $X'\times Y'$ is $\mathbb R$-factorizable, then 
$X'\times Y'$ is $z$-embedded in $\beta X'\times \beta Y'$ (because $(1)\Rightarrow(3)$) and hence 
$X\times Y$ is $z$-embedded in $\beta X'\times \beta Y'$. 

The implication $(3)\Rightarrow (6)$ is obvious. It remains to show that $(5)\Rightarrow (2)$. Let $U$ be a 
cozero set in $X\times Y$, and let $U'$ be a cozero set in an $\mathbb R$-factorizable product $X'\times Y'$ for 
which $U=U'\cap (X\times Y)$. Since (2) holds for $X'\times Y'$, it follows that $U'=\bigcup_{n\in 
\omega}(V_n\times W_n)$, where the $V_n$ and $W_n$ are cozero sets in $X'$ and $Y'$, respectively. Clearly, the 
$V_n\cap X$ and $W_n\cap Y$ are cozero sets in $X$ and $Y$. We have 
\begin{align*}
U&= \Bigl(\bigcup_{n\in \omega}(V_n\times W_n)\Bigr)\cap (X\times Y) = 
\bigcup_{n\in \omega}\bigl((V_n\times W_n)\cap (X\times Y)\bigr)\\
&=\bigcup_{n\in \omega}\bigl((V_n\cap X)\times (W_n\times Y)\bigr).
\end{align*}
Thus, (2) holds for $X\times Y$.
\end{proof}

The paper \cite{Blair-Hager} studied pairs of spaces $X$ and $Y$ satisfying condition (3) in 
Proposition~\ref{p2.1}. According to this proposition, (3) is equivalent to the 
$\mathbb R$-factorizability of the product $X\times Y$. In what follows, when referring to 
\cite{Blair-Hager}, we will bear in mind the equivalence $(1)\Leftrightarrow(3)$.

Given a cardinal $\kappa$, by $D(\kappa)$ we denote $\kappa$ with the discrete topology; thus, 
$D(\kappa)$ is a discrete space of cardinality~$\kappa$.

\begin{proposition}
\label{p2.2}
If a product $X\times Y$ of spaces is $\mathbb R$-factorizable and 
$Y$ contains a discrete family of open sets of cardinality $\kappa$, then 
the product $X\times D(\kappa)$ is $\mathbb R$-factorizable.
\end{proposition}

\begin{proof}
Let $\{ U_\alpha:\alpha<\kappa\}$ be a discrete family of nonempty open sets in $Y$.
We choose $y_\alpha\in U_\alpha$ for each $\alpha<\kappa$ and set $Q=\{y_\alpha:\alpha<\kappa\}$. Clearly, $Q$ is homeomorphic 
to~$D(\kappa)$.

We claim that $X\times Q$ is $C$-embedded in $X\times Y$. Indeed, let $f$ be a continuous function on 
$X\times Q$. For each $\alpha<\kappa$, we choose a continuous function $g_\alpha\colon Y\to [0,1]$ such that 
$g_\alpha(Y\setminus U_\alpha)=\{0\}$ and $g_\alpha(y_\alpha)=1$ and define a function $h_\alpha\colon X\times Y\to \mathbb R$ by 
setting $h_\alpha(x,y)=f(x, y_\alpha)\cdot g_\alpha(y)$ for $(x,y)\in X\times Y$. The function 
$h=\sum_{\alpha<\kappa} h_\alpha$ is a continuous extension of~$f$.

Since $X\times Q$ is $C$-embedded in $X\times Y$, it follows that $X\times Q$ is $z$-embedded in 
$X\times Y$. According to Proposition~\ref{p2.1}, the product $X\times Q$ is 
$\mathbb R$-factorizable. 
\end{proof}

\begin{proposition}
\label{p2.3}
For a space $Y$ and a cardinal $\kappa$, the following conditions are equivalent: 
\begin{enumerate}
\item[\rm(1)]
$X\times D(\kappa)$ is $\mathbb R$-factorizable;
\item[\rm(2)]
for any family $\{F_\alpha:\alpha<\kappa\}$ of zero sets in $X$, there exists a second-countable space $M$ and a 
continuous map $g\colon  X\to M$ such that $F_\alpha=g^{-1}(g(F_\alpha))$ and $g(F_\alpha)$ is closed in 
$M$ for each $\alpha<\kappa$; 
\item[\rm(3)] every continuous map $f\colon X\to Y$ to a space $Y$ with $w(Y)\leq\kappa$ factors through 
a continuous map to a second-countable space. 
\end{enumerate} 
\end{proposition} 

\begin{proof} 
Let us prove that $(1)\Rightarrow(2)$. For each $\alpha<\kappa$, we fix a  continuous function $f_\alpha$ on 
$X$ such that $f^{-1}_\alpha(\{0\})=F_\alpha$. The function 
\[ 
f\colon  X\times D(\kappa)\to\mathbb R,\quad (x,\alpha) 
\mapsto f_\alpha(x), 
\] 
is continuous. Condition (1) implies the existence of second-countable spaces $M$ and $E$ and continuous maps 
$g\colon X\to M$, $q\colon D(\kappa)\to E$, and $h\colon M\times E\to \mathbb R$ for which $f=h\circ (g\times 
q)$. For each $\alpha<\kappa$, $F_\alpha\subset g^{-1}(g(F_\alpha))$. On the other hand, if $x\in 
g^{-1}\bigl(g(F_\alpha)\bigr)$, then there exists a $y\in F_\alpha$ for which $g(y)=g(x)$ and 
$f\bigl((x,\alpha)\bigr)= h\bigl(g(x), q(\alpha)\bigr) = h\bigl(g(y), q(\alpha)\bigr) = 
f\bigl((y,\alpha)\bigr)=0$, which means that $f_\alpha(x)=0$ and $x\in F_\alpha$. Thus, $F_\alpha= 
g^{-1}(g(F_\alpha))$. Note that, for $x\in M$ and $\alpha\in \omega_1$, $h(x,q(\alpha))=0$ if and only if there 
exists a $z\in X$ for which $g(z)=x$ and $f(z, \alpha) = 0$, i.e., $z\in F_\alpha$ and $x\in g(F_\alpha)$. 
Therefore, $g(F_\alpha)\times \{q(\alpha)\}=h^{-1}(\{0\})\cap (M\times \{q(\alpha)\})$. This set is closed in 
$M\times \{q(\alpha)\}$, and hence $g(F_\alpha)$ is closed in~$M$.

To prove $(2)\Rightarrow(3)$, we take a base $\{ U_\alpha: \alpha<\kappa\}$ of $Y$ consisting of cozero sets and put 
$F_\alpha=f^{-1}(Y\setminus U_\alpha)$ for $\alpha<\kappa$. Let $M$ and $g\colon X\to M$ be as in (2). Note that if 
$x,y\in Y$ and $f(x)\ne f(y)$, then $g(x)\ne g(y)$. Indeed, there exists an $\alpha<\kappa$ for which $f(x)\in 
U_\alpha$ and $f(y)\notin U_\alpha$, that is, $x\notin F_\alpha$ and $y\in F_\alpha$. Since 
$g^{-1}(g(F_\alpha))=F_\alpha$, it follows that $g(x)\notin g(F_\alpha)$, while $g(y)\in g(F_\alpha)$. Therefore, 
choosing an arbitrary point $z'\in g^{-1}(z)$ and setting $h(z)=f(z')$ for every $z\in Z$, we 
obtain a well-defined map $h\colon Z\to X$. It is continuous, because $h^{-1}(U_\alpha)= M\setminus g(F_\alpha)$ 
and $g(F_\alpha)$ is closed by assumption. Clearly, $f=h \circ g$.

It remains to prove the implication $(3)\Rightarrow(1)$. Let $\varphi$ be a continuous function 
$X\times D(\kappa)\to \mathbb R$. For each $\alpha\in D(\kappa)$, we define a function $f_\alpha\colon 
X\to \mathbb R$ by $f_\alpha(x)=\varphi(x,\alpha)$ for $x\in X$ and set $f=\bigdiag_{\alpha<\kappa}f_\alpha\colon 
X\to \mathbb R^\kappa$. Condition $(3)$ implies the existence of a second-countable space $M$ and continuous maps 
$g\colon X\to M$ and $h\colon M\to \mathbb R^\kappa$ such that $f=h \circ g$. Let $\pi_\alpha\colon \mathbb 
R^\kappa\to \mathbb R$ denote the projection onto the $\alpha$th coordinate for $\alpha<\kappa$, and let 
$$ 
\psi\colon M\times D(\kappa)\to\mathbb R,\quad  (z,\alpha) \mapsto \pi_\alpha(h(z)). 
$$ 
The function $\psi$ is continuous, and $\varphi=\psi\circ(g\times \operatorname{id}_{D(\kappa)})$. Since $M$ is 
second-countable, it follows that the product $M\times D(\kappa)$ is $\mathbb R$-factorizable 
\cite[Theorem~3.2]{Blair-Hager}. Hence there exist second-countable spaces $M'$ and $S$ and continuous maps 
$i\colon M\to M'$, $q\colon D(\kappa)\to S$, and $\nu\colon M'\times S\to \mathbb R$ for which $\psi =\nu\circ 
(i\times q)$. Let $\mu = \nu\circ (i\times \operatorname{id}_S)$. Then $\mu$ is a continuous map $M\times S\to 
\mathbb R$ and $\psi = \mu\circ (\operatorname{id}_M\times q)$. We have 
$\varphi= \psi\circ(g\times \operatorname{id}_{D(\kappa)})=\mu\circ (\operatorname{id}_M\times q) 
\circ(g\times \operatorname{id}_{D(\kappa)}) = \mu\circ (g\times q)$.
\end{proof}

\begin{proposition}
\label{p2.4}
Let $\kappa$ be a cardinal, and let $X$ be a space with $w(X)\le \kappa$. 
If the product $X\times D(\kappa)$ is $\mathbb R$-factorizable, then $X$ is second-countable. 
\end{proposition}

\begin{proof}
Proposition~\ref{p2.3} implies the existence of a second-countable space $M$ and 
continuous maps $g\colon X\to M$ and $h\colon M\to X$ for which $\operatorname{id}_X=h \circ g$. Obviously, the maps $g$ and 
$h$ are homeomorphisms. 
\end{proof}

Proposition \ref{p2.4} strengthens Theorem 3.1 of~\cite{Blair-Hager}.

\begin{proposition}
\label{p2.5}
If $X\times D(\omega_1)$ is $\mathbb R$-factorizable, 
then so is $X^\omega\times D(\omega_1)$. 
\end{proposition}

To prove this proposition, we need a lemma. 
Given a map $f\colon X\to Y$, by $f^{\times\omega}$ we denote the product map $X^\omega\to Y^\omega$. 
Let us say that a subset $F$ of $X^\omega$ is a \emph{strong zero 
set} in $X^\omega$ if there exists a second-countable space $M$ and a continuous function $f\colon X\to M$ such 
that 
$$ 
F=(f^{\times\omega})^{-1}\bigl(f^{\times\omega}(F)\bigr)\quad\text{and}\quad \text{$f^{\times\omega}(F)$ 
is closed in $M^\omega$}. 
$$

\begin{lemma}
\label{l2.1}
If $X\times D(\omega_1)$ is $\mathbb R$-factorizable, then any closed subset of $X^\omega$ is a strong zero set. 
\end{lemma}

\begin{proof}
Suppose that there exists a closed set $F$ in $X^\omega$ which is not a strong zero set.  
To obtain a contradiction, we will recursively construct second-countable spaces $M_\beta$, continuous maps 
$f_\beta\colon X\to M_\beta$, and strictly decreasing strong zero sets $F_\beta$ for $\beta<\omega_1$ 
so that 
$$
F_\beta= (f_\beta^{\times \omega})^{-1}\bigl(\overline{f_\beta^{\times\omega}(F)}\bigr) 
\quad\text{and hence}\quad
f_\beta^{\times\omega}(F_\beta)=\overline{f_\beta^{\times\omega}(F)}
\eqno(*)
$$
for each $\beta<\omega_1$. We define $M_0$ to 
be a singleton and $f_0$ to be the map $X\to M_0$. For $F_0$ we take $X^\omega$. Suppose that $\alpha>0$ and 
$M_\beta$, $f_\beta$, and $F_\beta$ are defined for all $\beta<\alpha$. We set 
$$ 
f_\alpha^* = \diag_{\beta<\alpha}f_\beta\colon X\to \prod_{\beta<\alpha}M_\beta\quad\text{and}\quad 
F_\alpha^*= ({f_\alpha^*}^{\times\omega})^{-1}\bigl(\overline{{f_\alpha^*}^{\times \omega}(F)}\bigr). 
$$ 
From $(*)$ it follows that $F^*_\alpha\subset F_\beta$ for all $\beta<\alpha$. The set $F^*_\alpha$ satisfies 
condition $(*)$ and hence $F^*_\alpha\ne F$, because $F$ is not a strong zero set. Clearly, $F\subset 
F^*_\alpha$. Take
$$
(x_n)_{n<\omega}\in F^*_\alpha\setminus F\subset X^\omega.
$$
Since $F$ is closed in $X^\omega$, there exists an $N<\omega$ and cozero sets $U_i\subset X$, $i\le N$, 
such that 
$$ 
(x_n)_{n<\omega}\in U_1\times U_2\times \dots \times U_N\times X\times X\times X\times \dots \subset 
X^\omega\setminus F. 
$$
For each $n\le N$, let $g_n\colon X\to [0,1]$ be a continuous  function such that $g_n(x_n)=0$ and 
$g_n(X\setminus U_n)\subset \{1\}$. We set 
$$
M_\alpha = \mathbb R^N\times \prod_{\beta<\alpha} M_\beta,  \quad 
f_\alpha = \bigl(\bigdiag_{n\le N}g_n\bigr)\diag f_\alpha^*,  
 \quad \text{and}\quad
F_\alpha = (f_\alpha^{\times \omega})^{-1}\bigl(\overline{f_\alpha^{\times\omega}(F)}\bigr).
$$ 
Note that $(x_n)_{n<\omega}\notin 
F_\alpha$. Indeed, for each $i\le N$, the $i$th coordinate of 
$f_\alpha^{\times\omega}\bigl((x_n)_{n<\omega}\bigr)$ is 
$$
\bigl(\bigl(\bigdiag_{n\le N}g_n\bigr)\diag f_\alpha^*\bigr)(x_i) \in  
\mathbb R^{i-1}\times \{0\}\times \mathbb R^{N-i}\times \prod_{\beta<\alpha}M_\beta, 
$$
while for every $\bigl((y_n)_{n\in \omega}\bigr)\in F$, there exists an $i\le N$ for 
which $y_i\notin U_i$, so that the $i$th coordinate of $f_\alpha^{\times\omega}\bigl((y_n)_{n<\omega}\bigr)$ is 
$$ 
\bigl(\bigl(\bigdiag_{n\le N}g_n\bigr)\diag f_\alpha^*\bigr)(y_i) \in \mathbb R^{i-1}\times\{1\}\times \mathbb R^{N-i}\times 
\prod_{\beta<\alpha}M_\beta,
$$ 
whence 
\begin{align*}
\overline {f_\alpha^{\times\omega}(F)}&\subset 
\overline{\bigcup_{i\le N} M_\alpha^{i-1}\times \bigl(\mathbb R^{i-1}\times\{1\}\times \mathbb R^{N-i}\times 
\prod_{\beta<\alpha}M_\beta\bigr)\times M_\alpha\times M_\alpha\times \dots} \\
&= 
\bigcup_{i\le N} M_\alpha^{i-1}\times \bigl(\mathbb R^{i-1}\times\{1\}\times \mathbb R^{N-i}\times 
\prod_{\beta<\alpha}M_\beta\bigr)\times M_\alpha\times M_\alpha\times \dots\,.
\end{align*}

Thus, $F_\alpha\subsetneq F^*_\alpha\subset F_\beta$, $\beta<\alpha$. 

Having constructed $M_\alpha$, $f_\alpha$, and $F_\alpha$ for all $\alpha<\omega_1$, we set 
$$
Y=\prod_{\alpha<\omega_1}M_\alpha \quad\text{and}\quad 
f=\diag_{\alpha<\omega_1} f_\alpha\colon X\to Y.
$$ 
By Proposition~\ref{p2.3}\,(3) there exists a second-countable space $M$ and continuous maps $g\colon X\to 
M$ and $h\colon M\to Y$ for which $f=h\circ g$. We have $f^{\times\omega}=h^{\times\omega}\circ 
g^{\times\omega}$. Let 
$$
\varphi\colon Y^\omega=\Bigl(\prod_{\alpha<\omega_1}M_\alpha\Bigr)^\omega \to \prod_{\alpha <\omega_1} 
M_\alpha^\omega
$$
be the obvious homeomorphism permuting factors. Then 
$$
(\varphi\circ h^{\times\omega})\circ g^{\times\omega} =\varphi\circ f^{\times \omega}= 
\diag_{\alpha<\omega_1} f_\alpha^{\times\omega}\colon X^\omega\to 
\prod_{\alpha <\omega_1} M_\alpha^\omega.  
$$
Note that the sets 
$$
F'_\beta= \prod_{\gamma\le\beta}f_\gamma^{\times\omega}(F_\gamma)\times 
\prod_{\beta<\gamma<\omega_1}M_\gamma^\omega, \quad \beta<\omega_1,
$$ 
are closed in 
$\prod_{\gamma<\omega_1}M_\gamma^\omega$ and 
$F_\beta = \bigl(\diag_{\alpha<\omega_1} 
f_\alpha^{\times\omega}\bigr)^{-1}(F'_\beta)$. Since 
$F_\beta$ strictly decrease, it follows that $(\varphi\circ h^{\times\omega})^{-1}(F'_\beta)$, $\beta<\omega_1$, 
form a strictly decreasing sequence of closed sets in the second-countable space $M^\omega$. 
The complements to these sets form an uncountable open 
cover of the complement to their intersection having no 
countable subcover, which cannot exist, because $M$ is hereditarily Lindel\"of. This contradiction proves that 
any closed subset of $X^\omega$ is a strong zero set. \end{proof}

\begin{proof}[Proof of Proposition~\ref{p2.5}] It suffices to show that $X^\omega$ satisfies 
condition~(2) of Proposition~\ref{p2.3}. Let $\{F_\alpha:\alpha <\omega_1\}$ be a family of zero sets in 
$X^\omega$. By virtue of Lemma~\ref{l2.1}, for each $\alpha<\omega_1$, there exists a 
second-countable space $M_\alpha$ and a continuous map $f_\alpha\colon X\to M_\alpha$ for which 
$$ 
F_\alpha=(f_\alpha^{\times\omega})^{-1}\bigl(f_\alpha^{\times\omega}(F_\alpha)\bigr)\quad\text{and}\quad 
\text{$f_\alpha^{\times\omega}(F_\alpha)$ is closed in $M_\alpha^\omega$}. \eqno(**)
$$
Let $Y$, $f$, $M$, $g$, $h$, and $\varphi$ be defined as in the proof of Lemma~\ref{l2.1}. 
It follows from $(**)$ that, for each $\beta<\omega_1$, 
$$
F_\beta=(\bigdiag_{\alpha <\omega_1}f_\alpha^{\times\omega})^{-1}\bigl(\bigdiag_{\alpha <\omega_1}
f_\alpha^{\times\omega}(F_\beta)\bigr)
\quad\text{and}\quad
\bigdiag_{\alpha <\omega_1}f_\alpha^{\times\omega}(F_\beta)
\text{ is closed in }
\prod_{\alpha<\omega_1}M_\alpha^\omega,
$$
whence 
$$
F_\beta=(\varphi\circ f^{\times\omega})^{-1}\bigl((\varphi\circ f^{\times\omega})(F_\beta)\bigr)
\quad\text{and}\quad 
F_\beta = (g^{\times\omega})^{-1}\bigl(g^{\times\omega}(F_\beta)\bigr),
$$
because $\bigdiag_{\alpha 
<\omega_1}f_\alpha^{\times\omega} = \varphi\circ f^{\times\omega} = (\varphi\circ h^{\times\omega})\circ 
g^{\times\omega}$. From the same considerations it follows that $g^{\times\omega}(F_\beta) = (\varphi\circ 
h^{\times\omega})^{-1}\bigl(\bigdiag_{\alpha 
<\omega_1}f_\alpha^{\times\omega}(F_\beta)\bigr)$ and, therefore, $g^{\times\omega}(F_\beta)$ is closed in 
$X^\omega$. Thus, condition (2) of Proposition~\ref{p2.3} does hold for $\{F_\alpha:\alpha <\omega_1\}$ (with 
$M^\omega$ and $g^{\times\omega}$ playing the roles of $M$ and~$g$). \end{proof}

\begin{proposition}
\label{p2.6}
If $X\times D(\omega_1)$ is $\mathbb R$-factorizable, then 
$X^\omega$ is hereditarily 
Lindel\"of and hereditarily separable. 
\end{proposition}

\begin{proof} 
In view of Proposition~\ref{p2.5}, it suffices to prove that $X$ is hereditarily Lindel\"of and 
hereditarily separable.
 
First, we show that $X$ is hereditarily Lindel\"of. Suppose it is not. Then there exists a right separated 
set in $X$, i.e., a subspace $R=\{x_\alpha:\alpha<\omega_1\}\subset X$ in which all initial 
segments $\{x_\beta:\beta<\alpha\}$, $\alpha<\omega_1$, are open \cite{Juhasz}. For each $\beta<\omega_1$, 
let $U_\beta$ be a cozero neighborhood of $x_\beta$ in $X$ such that $U_\alpha\cap R\subset 
\{x_\gamma:\gamma<\beta+1\}$. Then the sets $F_\alpha = X\setminus\bigcup_{\beta<\alpha}U_\beta$, 
$\alpha<\omega_1$, are zero sets in $X$, and they strictly decrease, because $\alpha\in 
F_\alpha\setminus F_{\alpha+1}$ for each $\alpha<\omega_1$. Proposition~\ref{p2.4} implies the 
existence of a second-countable space $M$ and a continuous map 
$g\colon X\to M$ such that $F_\alpha=g^{-1}(g(F_\alpha))$ and $g(F_\alpha)$ is closed in $M$ for each $\alpha<\omega_1$. 
Therefore, $\bigl(g(F_\alpha)\bigr)_{\alpha<\omega_1}$ is a strictly decreasing $\omega_1$-sequence of closed 
sets in the second-countable space $M$ and $\{M\setminus g(F_\alpha):\alpha < \omega_1\}$ is an uncountable open cover 
of $M\setminus \bigcup_{\alpha<\omega_1} g(F_\alpha)$ containing no countable subcover, which cannot exist, 
because $M$ is hereditarily Lindel\"of.

Thus, $X$ is hereditarily Lindel\"of and, therefore, perfectly normal.

Let us prove that $X$ is hereditarily separable. Suppose it is not. 
Then there exists a left separated set in 
$X$, i.e., a set $L=\{x_\alpha:\alpha<\omega_1\}\subset X$ such that $\{x_\beta:\beta<\alpha\}$ is closed in 
$L$ for each $\alpha<\omega_1$ \cite{Juhasz}. Clearly, the closed subsets 
$F_\alpha=\overline{\{x_\beta:\beta<\alpha\}}$ of $X$ strictly increase. They are zero sets, because $X$ is 
perfectly normal. According to Proposition \ref{p2.4}, there exists 
a second-countable space $M$ and a continuous map $g\colon X\to M$ such that $F_\alpha=g^{-1}(g(F_\alpha))$ and 
$g(F_\alpha)$ is closed in $M$ for each $\alpha<\omega_1$. Thus, $\bigl(g(F_\alpha)\bigr)_{\alpha<\omega_1}$ is a strictly 
increasing $\omega_1$-sequence of closed subsets of $M$ and $M'=\bigcup_{\alpha<\omega_1}g(F_\alpha)$ is a nonseparable 
subspace of $M$, because any countable subset of $M'$ is contained in $g(F_\alpha)$ for some $\alpha<\omega$. 
Such an $M'$ cannot exists, since $M$ is hereditarily 
separable. 
\end{proof}

As is known, any space $X$ whose square is hereditarily Lindel\"of is submetrizable (see, e.g., 
\cite[Lemma~8.2]{Borges}). This implies the following assertion. 

\begin{corollary}
Any space $X$ for which $X\times D(\omega_1)$ is $\mathbb R$-factorizable is submetrizable. 
\end{corollary}

\begin{proposition}
\label{p2.7}
Any space $X$ for which the product $X\times D(2^\omega)$ is $\mathbb R$-factorizable is second-countable.
\end{proposition} 

\begin{proof} It follows from Proposition~\ref{p2.6} that $X$ is separable. 
Therefore, $w(X)\leq 2^\omega$ and $X$ is second-countable by Proposition~\ref{p2.4}. 
\end{proof}

\begin{corollary}[$\mathrm{CH}$]
\label{c2.2} 
Any space $X$ for which the product $X\times D(\omega_1)$ is $\mathbb R$-factorizable is second-countable. 
\end{corollary}

As mentioned in the introduction, the CH assumption cannot be omitted, because if $\mathfrak 
b>\omega_1$, then the product of the countable Fr\'echet--Urysohn fan and $D(\omega_1)$ is $\mathbb 
R$-factorizable. 

\begin{theorem}
\label{t2.1}
If a product $X\times Y$ of spaces is $\mathbb R$-factorizable and 
$Y$ is not pseudo-$\aleph_1$-compact, then 
\begin{enumerate}
\item[\textup{(1)}]
$X^\omega$ is hereditarily Lindel\"of and hereditarily separable;
\item[\textup{(2)}]
if $w(X)\leq \omega_1$, then $X$ is second-countable;
\item[\textup{(3)}]
under $\mathrm{CH}$,
$X$ is second-countable.
\end{enumerate}
\end{theorem}

\begin{proof}
It follows from Proposition \ref{p2.2}  that $X\times D(\omega_1)$  is $\mathbb R$-factorizable.
Hence Propositions~\ref{p2.6} and \ref{p2.4} imply (1) and (2) and 
Corollary~\ref{c2.2} implies~(3). 
\end{proof}

\section{$\mathbb R$-Factorizable Products of Topological Groups} 
\label{s3}

\begin{theorem}
\label{t3.1}
For topological groups $G$ and $H$, the following conditions are equivalent:
\begin{enumerate}
\item[\textup{(1)}]
$G\times H$ is an $\mathbb R$-factorizable group;
\item[\textup{(2)}]
the groups $G$ and $H$ are $\mathbb R$-factorizable and the product $G\times H$ is multiplicatively 
$\mathbb R$-factorizable. 
\end{enumerate} 
\end{theorem} 

\begin{proof} 
First, we show that $(1)\Rightarrow(2)$. The  projections of 
$G\times H$ onto the factors are open homomorphisms; hence by Theorem~\ref{t1.2} the groups $G$ and $H$ 
are $\mathbb R$-factorizable. According to Theorem~\ref{tB}, the product $G\times H$ is 
multiplicatively $\mathbb R$-factorizable.

Let us prove that $(2)\Rightarrow(1)$. The group $G$  is $\omega$-narrow, being $\mathbb R$-factorizable 
\cite[Proposition~8.1.3]{AT}. According to a theorem of Guran (see \cite[Theorem~3.4.23]{AT}), $G$ is 
a topological subgroup of a product $G'$ of second-countable topological groups. Since $G$ is $\mathbb R$-factorizable, 
it follows that $G$ is $z$-embedded in $G'$~\cite[Theorem~8.2.7]{AT}. Similarly, $H$ is $z$-embedded in a 
product $H'$ of second-countable topological groups. Hence 
$G\times H$ is $z$-embedded in $G'\times H'$ by Proposition~\ref{p2.1}. Since $G'\times H'$ 
is a product second-countable topological groups, it follows by Theorem~8.1.14 of \cite{AT} that $G'\times H'$ 
is an $\mathbb R$-factorizable group. 
According to \cite[Theorem~8.2.7]{AT}, the group $G\times H$ is $\mathbb R$-factorizable, because 
it is $z$-embedded in the $\mathbb R$-factorizable group $G'\times H'$. 
\end{proof}

\begin{proposition}
\label{p3.1}
If a product $G\times H$ of  topological groups is multiplicatively $\mathbb R$-factorizable, then either 
both groups $G$ and $H$ are $\omega$-narrow or one of them is $\mathbb R$-fac\-tor\-iz\-able. 
\end{proposition} 

\begin{proof} 
Theorem~\ref{t2.1} implies that either both groups $G$ and $H$ are 
pseudo-$\aleph_1$-compact or one of them is Lindel\"of. It remains to recall that any 
pseudo-$\aleph_1$-compact group is $\omega$-narrow~\cite[3.4.31]{AT} and any Lindel\"of group is 
$\mathbb R$-fac\-tor\-iz\-able~\cite[8.1.6]{AT}. 
\end{proof}

\begin{question}
\label{q5}
Is it true that if a product  $G\times H$ of topological groups is multiplicatively $\mathbb R$-factorizable,  
then one of the groups $G$ and $H$ is $\mathbb R$-factorizable? 
\end{question}

When the product $G\times H$ is a square, the answer to this 
question is positive even in the much more general case of topological universal algebras. 

Recall that  an \emph{$n$-ary  operation} on a set $X$ is any map from $X^n$ to $X$. 
A {\em universal algebra} is a nonempty set $X$ together with a set of operations on $X$. If $X$ is endowed with 
a topology and all of these operations are continuous, then $X$ is called a {\em topological universal algebra}. 
The operations are indexed by the elements of a set $\Sigma$ of symbols of operations. This set $\Sigma$, 
together with a map $\nu\colon \Sigma\to \omega$ assigning arity to each $\sigma\in \Sigma$, is called a 
\emph{signature}. A universal algebra with a signature $\Sigma$ is called a \emph{$\Sigma$-algebra}. For 
$\sigma \in \Sigma$, the operation on a $\Sigma$-algebra with index $\sigma$ is usually denoted by the same 
symbol $\sigma$. Groups are universal algebras with signature $\Sigma_{\text{gr}}=\{e,{}^{-1},\cdot\}$, where $e$ 
is the symbol of a nullary operation (which is identified with the identity element of a group), $^{-1}$ is the 
symbol of a unary operation (inversion), and $\cdot$ is the symbol of a binary operation (multiplication). 
Topological groups are topological $\Sigma_{\text{gr}}$-algebras. 

Let $\Sigma$ be a signature. A map $\varphi\colon X\to Y$ of $\Sigma$-algebras is called a \emph{homomorphism} if, 
for any $n\in \omega$, any symbol $\sigma\in \Sigma$ of an $n$-ary operation, and any $x_1,x_2,\dots,x_n\in X$, 
$\varphi(\sigma(x_1,x_2,\dots,x_n))=\sigma(\varphi(x_1),\varphi(x_2),\dots,\varphi(x_n))$.

\begin{definition}[\cite{EAR}]
\label{d3.1}
Let $\Sigma$ be a signature.
A topological $\Sigma$-algebra $X$ is said to be \emph{$\mathbb R$-factorizable} if, given any continuous function 
$f\colon X\to\mathbb R$, there exists a second-countable topological $\Sigma$-algebra $Y$, a continuous 
homomorphism $g\colon X\to Y$, and a continuous function $h\colon Y\to\mathbb R$ such that $f=h\circ g$. 
\end{definition}

Formally, Definition~\ref{d3.1} as applied to a topological group treated as a universal 
algebra differs from the definition of the $\mathbb R$-factorizability of topological groups, because 
Definition~\ref{d3.1} does not require $Y$ to be a topological group, it only 
requires it to be a topological $\Sigma_{\text{gr}}$-algebra, that is, a set with a nullary operation $e$, 
a unary operation ${}^{-1}$, and a binary operation $\cdot$ on which no constraints (like 
associativity of multiplication and the familiar properties of $e$ and ${}^{-1}$) are imposed. 
But in fact these definitions coincide, because it is easy to see that any homomorphic image of a group is 
a group and, therefore, $Y$ is a topological group. Similar considerations apply to topological 
semigroups and paratopological groups.

The most studied and interesting case is that of universal algebras with finite signature. In \cite{EAR} the 
general case was considered, in which the signature has arbitrary cardinality and is endowed with a 
topology. In what follows, we consider universal algebras with finite signature, in which case 
the signature is a finite discrete space.

In \cite{EAR} the notion of an \emph{$\mathbb R$-factorizable product $X^n$ over a signature $\Sigma$} was 
introduced. In the case of a finite discrete signature, this notion coincides with 
that of a multiplicatively $\mathbb R$-factorizable product $X^n$. The following proposition follows from 
Theorem~12 in~\cite{EAR}.

\begin{proposition}
\label{p3.2}
Let $X$ be a topological $\Sigma$-algebra with finite signature $\Sigma$ such that the arities of all  
operations $\sigma\in \Sigma$ do not exceed $n$. If the product $X^n$ is multiplicatively $\mathbb R$-factorizable, 
then $X$ is an $\mathbb R$-factorizable topological $\Sigma$-algebra. 
\end{proposition}

\begin{corollary}
\label{c3.1}
If the square $G\times G$ of a topological group $G$ is multiplicatively $\mathbb R$-factorizable,
then the group $G$ is $\mathbb R$-factorizable.
\end{corollary}

\begin{theorem}
\label{t3.2}
Let $G$ be a topological group. The product $G\times G$ is multiplicatively $\mathbb R$-factorizable 
if and only if the group $G\times G$ is $\mathbb R$-factorizable. 
\end{theorem}

\begin{proof}
Suppose that the product $G\times G$ is multiplicatively $\mathbb R$-factorizable. 
Then the group $G$ is $\mathbb R$-factorizable by Corollary~\ref{c3.1}. According to Theorem~\ref{t3.1}, 
the group $G\times G$ is $\mathbb R$-factorizable as well.

Conversely, if the group $G\times G$ is $\mathbb R$-factorizable, then the product $G\times G$ is multiplicatively 
$\mathbb R$-factorizable by Theorem~\ref{tB}. 
\end{proof}

\begin{corollary}
Let $G$ be a topological group homeomorphic to a dense subspace of a product $X$ of Lindel\"of 
$\Sigma$-groups. Then $G^\kappa$ is $\mathbb R$-factorizable for any cardinal~$\kappa$. 
\end{corollary}

\begin{proof}
According to Theorem~8.4.6 of \cite{AT}, any product of Lindel\"of 
$\Sigma$-groups, in particular $X^\kappa\times X^\kappa$, is perfectly $\varkappa$-normal. Since $G^\kappa\times 
G^\kappa$ is homeomorphic to a dense subspace of $X^\kappa\times X^\kappa$, it follows by Theorem~5.1 of 
\cite{Blair} that $G^\kappa\times G^\kappa$ is $z$-embedded in $X^\kappa\times X^\kappa$. By 
Theorem~\ref{t3.2} $X^\kappa\times X^\kappa$ is multiplicatively $\mathbb R$-factorizable, because any product 
of Lindel\"of $\Sigma$-groups is $\mathbb R$-factorizable \cite[Theorem~8.1.14]{AT}. Therefore, so is 
$G^\kappa\times G^\kappa$. It remains to apply Corollary~\ref{c3.1}. 
\end{proof}

Proposition~\ref{p3.2} implies also the following statements.

\begin{corollary}
\label{c3.3}
Let $G$ be a topological semigroup. If the product $G\times G$ is multiplicatively $\mathbb R$-factorizable,
then the semigroup $G$ is $\mathbb R$-factorizable.
\end{corollary}

Recall that a \emph{paratopological group} is a group with a topology with respect to which multiplication is 
continuous. We say that a paratopological group is $\mathbb R$-factorizable if, given any continuous function 
$f\colon G\to \mathbb R$, there exists a second-countable paratopological group $H$, a continuous 
homomorphism $h\colon G \to H$, and a continuous function $g\colon H\to \mathbb R$ for which $f = g \circ h$. 

\begin{corollary}
\label{c3.4}
Let $G$ be a paratopological group. If the product $G\times G$ is multiplicatively $\mathbb R$-factorizable,
then $G$ is $\mathbb R$-factorizable. 
\end{corollary}

\begin{proof}
Let $f\colon G\to \mathbb R$ be a continuous function. Since $G$ is a topological semigroup, by 
Corollary~\ref{c3.3} there exists a second-countable topological semigroup $H$, a continuous 
homomorphism $h\colon G \to H$, and a continuous function $g\colon H\to \mathbb R$ for which $f = g \circ h$. We 
can assume that $h$ is surjective.  Since $h(x\cdot y)=h(x)\cdot h(y)$, it follows that, for any $y\in H$ and any 
$x\in h^{-1}(y)$, we have $h(x^{-1})\cdot y = y\cdot h(x^{-1}) = h(e)$ and $h(e)\cdot y=y\cdot h(e)=y$. 
Therefore, $H$ is a group with identity element $h(e)$ and inversion $h(x)^{-1}=h(x^{-1})$ (recall that $h$ is a 
surjection). Thus, $H$ is a paratopological group and $h$ is a group homomorphism. 
\end{proof}

\begin{corollary}
\label{c3.5}
Let $G$ be a paratopological group. If the product $G\times G$ is Lindel\"of, then the paratopological 
group $G$ is $\mathbb R$-factorizable. 
\end{corollary}

\begin{proof}
By Proposition~\ref{p2.1}\,(6) $G\times G$ is multiplicatively $\mathbb R$-factorizable, and by 
Corollary~\ref{c3.4} $G$ is $\mathbb R$-factorizable.
\end{proof}

\begin{corollary}
If $G$ is a topological group such that the group $G\times G$ is $\mathbb R$-factorizable and a topological 
group $H$ is homeomorphic to $G$, then the group $H\times H$ is $\mathbb R$-factorizable. 
\end{corollary}

\begin{proof}
If the product $G\times G$ is productively $\mathbb R$-factorizable, then the group $G\times G$ is 
$\mathbb R$-factorizable by Theorem~\ref{t3.2}. The group $G$ is the image of $G\times G$ under the natural 
projection homomorphism, which is continuous and open. Therefore, $G$ is $\mathbb R$-factorizable by 
Theorem~\ref{t1.2}. 
\end{proof}

\section{Weird $\mathbb R$-Factorizable Groups}
\label{s4}

Theorems~\ref{t2.1} and \ref{t3.1} have the following immediate consequence. 

\begin{theorem}
\label{t4.1}
If $G$ is a weird $\mathbb R$-factorizable group, $H$ is a topological group, and the group $G\times H$ is 
$\mathbb R$-factorizable, then 
\begin{enumerate}
\item[\rm(1)]
$H^\omega$ is hereditarily Lindel\"of and hereditarily separable; 
\item[\rm(2)]
if $w(H) \le \omega_1$, then $H$ is second-countable; 
\item[\rm(3)] under $\mathrm{CH}$, $H$ is second-countable. 
\end{enumerate}
\end{theorem}

\subsection{Squares of $\mathbb R$-Factorizable Groups}

Theorems~\ref{t3.2} and~\ref{t2.1} imply the following statement.

\begin{theorem}
\label{t4.2}
If $G$ is a topological group and the group $G\times G$ is $\mathbb R$-factorizable, then $G$ is 
pseudo-$\aleph_1$-compact.
\end{theorem}

Indeed, by Theorem~\ref{t3.2} $G\times G$ is multiplicatively $\mathbb R$-factorizable and by 
Theorem~\ref{t2.1} if $G$ is not pseudo-$\aleph_1$-compact, then it must be hereditarily Lindel\"of, which 
is impossible.   

\begin{corollary}
\label{c4.1}
The square of a weird $\mathbb R$-factorizable group is never $\mathbb R$-factorizable.
\end{corollary}

\subsection{$\mathbb R$-Factorizable Groups of Uncountable Pseudocharacter}

\begin{theorem}
\label{t4.3}
Any $\mathbb R$-factorizable group of uncountable pseudocharacter is pseudo-$\aleph_1$-compact. 
\end{theorem}

\begin{corollary}
\label{c4.2}
Any weird $\mathbb R$-factorizable group has countable pseudocharacter, i.e., is submetrizable.
\end{corollary}

The proof of Theorem~\ref{t4.3} is based on the following lemma. 

\begin{lemma}
\label{l4.1}
Suppose given a discrete family $\{U_\alpha:\alpha<\omega_1\}$ of cozero sets in a topological group $G$ and 
nonempty zero sets $F_\alpha\subset U_\alpha$, $\alpha<\omega_1$. Then there exists a 
second-countable group $H$ and a continuous homomorphism $h\colon G\to H$ such that $h(F_\alpha)$ is closed and 
$h^{-1}(h(F_\alpha))=F_\alpha$ for each $\alpha<\omega_1$. 
\end{lemma}

\begin{proof}
First, we choose sets $Q_n\subset \omega_1$, $n\in \omega$, so that, for every $\alpha\in \omega_1$, 
$$
\bigcap\{Q_n: \alpha\in Q_n\}=\{\alpha\}.
$$ 
As such sets we can take any countable base of any second-countable topology on $\omega_1$. Let 
$$ 
S_n = \bigcup\{F_\alpha: \alpha\in Q_n\},\quad n\in \omega.
$$
Since all $F_\alpha$ are pairwise disjoint, it follows that, for each $\alpha<\omega_1$, $F_\alpha\cap S_n\notin 
\varnothing$ if and only if $\alpha\in Q_n$, in which case $F_\alpha\subset S_n$. Moreover, if $\beta\ne \alpha$, 
then there exists an $n\in \omega$ such that $\alpha \in Q_n$ and $\beta\notin Q_n$, that is, $F_\alpha\subset  
S_n$ and $F_\beta\cap S_n=\varnothing$. Therefore, 
$$
F_\alpha = \bigcap \{S_n: \alpha\in Q_n\}.
$$ 

For each $\alpha<\omega_1$, we fix a continuous function $f_\alpha\colon G\to \mathbb R$ such that 
$f_\alpha^{-1}(\mathbb R\setminus \{0\})= U_\alpha$ and $f_\alpha^{-1}(1)=F_\alpha$. We set 
$$ 
\varphi_n=\sum_{\alpha\in Q_n} f_\alpha\colon G\to \mathbb R, \qquad n\in \omega. 
$$ 
Every function $\varphi_n$ is well defined and continuous, because the family $\{U_\alpha:\alpha\in Q_n\}$ is 
discrete; hence there exists a second-countable group $H_n$, a continuous homomorphism 
$h_n\colon G\to H_n$, and a continuous function $g_n\colon H_n\to \mathbb R$ for which $\varphi_n=g_n\circ h_n$. Note 
that $S_n=\varphi_n^{-1}(\{1\})=h_n^{-1}(g_n^{-1}(\{1\}))$. 

We set 
$$
h=\bigdiag_{n\in \omega} h_n\colon G\to \prod_{n\in \omega}H_n
\quad \text{and}\quad
H = h(G).
$$
Since $F_\alpha=\bigcap\{S_n: \alpha\in Q_n\}$, it follows that $x\in F_\alpha$ if and only if $h_n(x)\in 
g_n^{-1}(1)$ whenever $\alpha\in Q_n$. Thus, 
$$
F_\alpha= h^{-1}\bigl(\{(x_n)_{n\in \omega}\in H: x_n\in g_n^{-1}(1)\text{ for all $n\in \omega$ such that $\alpha 
\in Q_n$}\}\bigr).
$$ 
It follows from the continuity and surjectivity of $h$ that $h(F_\alpha)$ is closed in $H$ and 
$h^{-1}(h(F_\alpha))= F_\alpha$. 
\end{proof}

\begin{proof}[Proof of Theorem~\ref{t4.3}] 
Let $G$ be an $\mathbb R$-factorizable group of uncountable pseudocharacter with identity element $e$. Suppose that 
$G$ is not pseudo-$\aleph_1$-compact and let $\{U_\alpha:\alpha<\omega_1\}$ be a discrete family of nonempty 
cozero sets. For each $\alpha<\omega_1$, we choose $x_\alpha\in U_\alpha$ and a zero set $Z_\alpha\subset 
U_\alpha$ so that 
\begin{enumerate}
\item[(i)]
$e\in Z_\alpha\subset x_\alpha^{-1}\cdot U_\alpha$; 
\item[(ii)]
$Z_\alpha\subset Z_\beta$ and $Z_\alpha\ne Z_\beta$ if $\beta<\alpha$.
\end{enumerate}
This can easily be done by transfinite recursion as follows. For $Z_0$ we take any zero set containing $e$ and 
contained in $x_0^{-1}\cdot U_0$. Assuming that $\alpha>0$ and $Z_\beta$ are defined for $\beta<\alpha$, 
we choose any zero set $F'_\alpha\subset U_\alpha$ containing $x_\alpha^{-1}$ and put $Z'_\alpha = 
(x_\alpha^{-1}\cdot F'_\alpha)\cap \bigcap_{\beta<\alpha}Z_\beta$. Since $Z'_\alpha$ is a $G_\delta$-set and the 
pseudocharacter of $G$ is uncountable, it follows that there exists an $x\in Z'_\alpha\setminus \{e\}$; for 
$Z_\alpha$ we take the intersection of $Z'_\alpha$ with any zero set containing $e$ and not containing~$x$. 

Let $F_\alpha=x_\alpha \cdot Z_\alpha$ for $\alpha<\omega_1$. Then the sets $U_\alpha$ and $F_\alpha$ satisfy the 
assumptions of Lemma~\ref{l4.1}. Hence there exists a second-countable group $H$ and a continuous 
homomorphism $h\colon G\to H$ such that 
\begin{enumerate}
\item[(i)]
all $h(F_\alpha)$ are closed and hence so are all 
$$
h(Z_\alpha)=h(x_\alpha^{-1}\cdot F_\alpha)= h(x_\alpha^{-1})\cdot h(F_\alpha);$$
\item[(ii)]
$h^{-1}(h(F_\alpha))= F_\alpha$.
\end{enumerate}
Let $N$ be the kernel of $h$. Then $F_\alpha=N\cdot F_\alpha$ and 
$$
h^{-1}(h(Z_\alpha))=h^{-1}(x_\alpha^{-1}\cdot F_\alpha)=N\cdot x_\alpha^{-1}\cdot F_\alpha = x_\alpha^{-1}\cdot 
F_\alpha = Z_\alpha.
$$ 
Therefore, $h(Z_\alpha)\ne h(Z_\beta)$ for $\alpha\ne \beta$. 

Thus, the subsets $h(Z_\alpha)$, $\alpha<\omega_1$, of $H$ are closed and strictly decrease. Their complements 
form an uncountable open cover containing no countable subcover in the second-countable space $H\setminus 
\bigcap_{\alpha<\omega_1}h(Z_\alpha)$. This contradiction shows that the family $\{U_\alpha: \alpha<\omega_1\}$ 
cannot exist and $G$ is pseudo-$\aleph_1$-compact.
\end{proof}

\begin{corollary}
\label{c4.3}
If $G$ and $H$ are topological groups at least one of which is of uncountable pseudocharacter and the group 
$G\times H$ is $\mathbb R$-factorizable, then both $G$ and $H$ are $\mathbb R$-factorizable and pseudo-$\aleph_1$-compact. 
\end{corollary}

\begin{proof}
According to Theorem~\ref{t4.3}, the group $G\times H$ is pseudo-$\aleph_1$-compact; hence so are its 
images $G$ and $H$ under continuous open projection homomorphisms. By Theorem~\ref{t1.2} 
they are also $\mathbb R$-factorizable. 
\end{proof}

It follows from Theorem~\ref{t1.1} that every nonmetrizable compact group has 
uncountable pseudocharacter, because compact spaces do not admit strictly coarser Hausdorff topologies. 
Therefore, Corollary~\ref{c4.3} implies the following theorem of~\cite{AT}.

\begin{corollary}[{\cite[Theorem~8.5.11]{AT}}]
\label{c4.4}
If $G$ is a topological group, $H$ is a nonmetrizable compact topological group, and the group $G\times H$ is 
$\mathbb R$-factorizable, then $G$ is pseudo-$\aleph_1$-compact. 
\end{corollary}

\begin{corollary}
If an $\mathbb R$-factorizable group $G$ contains a nonmetrizable compact subspace, then $G$ is 
pseudo-$\aleph_1$-compact.
\end{corollary}

\begin{proof}
Let $K$ be a compact subspace of $G$. If $G$ is not pseudo-$\aleph_1$-compact, then by 
Corollary~\ref{c4.2} it is submetrizable and hence so is $K$. Since compact spaces do not admit 
strictly coarser Hausdorff topologies, it follows that $K$ is metrizable. 
\end{proof}

\subsection{$\mathbb R$-Factorizable Groups of Regular Uncountable Weight}

It follows from Theorems~8.5.2 and~8.5.8 of \cite{AT} that any $\mathbb R$-factorizable group $G$ with 
$w(G)^\omega<2^{\omega_1}$ is pseudo-$\aleph_1$-compact. Therefore, under the assumption $2^\omega<2^{\omega_1}$, 
any $\mathbb R$-factorizable group of weight $\omega_1$ is pseudo-$\aleph_1$-compact. In this section, we show that 
this assumption can be removed. Moreover, we prove that any $\mathbb R$-factorizable group of regular uncountable 
weight $\kappa$ is pseudo-$\kappa$-compact. 

We begin with a simple observation.

\begin{remark}
\label{r4.1}
Any $\mathbb R$-factorizable group $G$ of weight $\kappa$ embeds in a 
product of $\kappa$-many second-countable groups as a subgroup. 

Indeed, since $G$ is Tychonoff, its topology 
has a base $\{B_\alpha: \alpha<\kappa\}$ consisting of cozero sets. Continuous functions $f_\alpha\colon 
G\to \mathbb R$ witnessing that the $B_\alpha$ are cozero separate points from closed sets, and each $f_\alpha$ 
factors through a continuous homomorphism $h_\alpha\colon G\to H_\alpha$ to a second-countable group $H_\alpha$. 
Clearly, the homomorphisms $h_\alpha$ separate points from closed sets as well, so that the diagonal 
$\bigdiag_{\alpha< \kappa} h_\alpha\colon G\to \prod_{\alpha<\kappa}H_\alpha$ is a topological 
isomorphic embedding. 
\end{remark}

\begin{theorem}
\label{t4.4}
Any $\mathbb R$-factorizable group of regular uncountable weight $\kappa$ is 
pseudo-$\kappa$-compact. \end{theorem}

\begin{corollary}
If $G$ is a weird $\mathbb R$-factorizable group, then $w(G)>\omega_1$. 
\end{corollary}

\begin{proof}[Proof of Theorem~\ref{t4.4}]
Suppose that $\kappa$ is a regular uncountable cardinal and an $\mathbb R$-factorizable group $G$ of weight $\kappa$ is 
not pseudo-$\kappa$-compact, i.e., contains a discrete family $\{U_\alpha:\alpha<\kappa\}$ of nonempty cozero 
sets. Take $y_\alpha\in U_\alpha$ for each $\alpha\in \kappa$ and let $Y=\{y_\alpha:\alpha<\kappa\}$. The set $Y$ 
is $C$-embedded in $G$ (by Remark~\ref{r1.3}) and $G$ is $z$-embedded in a product 
$\prod_{\alpha<\kappa}H_\alpha$ of second-countable groups $H_\alpha$ (by Remark~\ref{r4.1} and 
Theorem~8.2.7 of \cite{AT}); hence every set $P\subset Y$ is cozero in $\prod_{\alpha<\kappa}H_\alpha$ (because 
any such set is cozero in the discrete space $Y$). 

Choose a countable base $\mathscr B_\alpha$ of the topology of $H_\alpha$ for each $\alpha<\kappa$. 
Recall that the standard base $\mathscr B$ of the topology of $\prod_{\alpha<\kappa}H_\alpha$ consists of sets 
of the form $\prod_{\alpha<\kappa} U_\alpha$, where $U_\alpha=H_\alpha$ for all but finitely many 
$\alpha<\kappa$ and $U_\alpha\in \mathscr B_\alpha$ for the remaining $\alpha<\kappa$. Clearly, $|\mathscr 
B|\le \kappa$; since $w(G)=\kappa$, it follows that $|\mathscr B|= \kappa$. Let us index the elements 
of $\mathscr B$ by ordinals: $\mathscr B=\{B_\alpha:\alpha<\kappa\}$. For each $\alpha<\kappa$, we set 
$P_\alpha = B_\alpha\cap Y$.

\begin{lemma}
\label{l4.2}
For any $M\subset Y$, there exists a countable set $C\subset \kappa$ such that $M=\bigcup_{\alpha\in 
C}P_\alpha$.
\end{lemma}

\begin{proof}
Let $M\subset Y$. There exists a continuous function $f\colon 
\prod_{\alpha<\kappa}H_\alpha\to \mathbb R$ such that $M=f^{-1}(\mathbb R\setminus \{0\})\cap Y$. It is well known that any 
real-valued continuous function on a product of separable spaces depends on only countably many coordinates 
(see, e.g., \cite{Ross-Stone}). This means that there exists a countable set $A\subset \kappa$ and a continuous 
function $g\colon \prod_{\alpha\in A}H_\alpha \to \mathbb R$ for which $f=g\circ\pi_A$ (we use the standard notation 
$\pi_A$ for the projection $\prod_{\alpha\in\kappa}H_\alpha \to \prod_{\alpha\in A}H_\alpha$). Thus, 
$$ 
f^{-1}(\mathbb R\setminus \{0\})= g^{-1}(\mathbb R\setminus \{0\})\times \prod_{\alpha\in\kappa\setminus A}H_\alpha. 
$$ 
The open set $g^{-1}(\mathbb R\setminus \{0\})$ in the countable product $\prod_{\alpha\in A}H_\alpha$ is a countable 
union of elements of the standard base of this product. Clearly, if $U$ is any such element, then $U\times 
\prod_{\alpha\in\kappa\setminus A}H_\alpha$ is an element of the standard base for the product 
$\prod_{\alpha\in\kappa}H_\alpha$, i.e., $U=B_\alpha$ for some $\alpha<\kappa$. This immediately implies 
the required assertion. 
\end{proof} 

In what follows, we identify $Y$ with $\kappa$; this can be done, e.g., by means of the bijection 
$y_\alpha\mapsto \alpha$. 

Thus, if $G$ is not pseudo-$\kappa$-compact, then there must exist sets $P_\alpha\subset \kappa$, 
$\alpha<\kappa$, such that any $A\subset \kappa$ is the union of fewer than $\kappa$ of them. Our goal is to show 
that this is impossible. 

\begin{lemma}
\label{l4.3}
Suppose that sets $P_\alpha\subset \kappa$, $\alpha<\kappa$, are such that, for every 
$\alpha<\kappa$, there exist ordinals $x$ and $y$ and a set $M\subset [x, y]$
 satisfying the following conditions:
\begin{enumerate}
\item[\textup{(i)}]
$\alpha<x<y<\kappa$\textup;
\item[\textup{(ii)}]
for any $C\subset \alpha$, 
$$
M\ne \bigcup_{\beta\in C}P_\beta\cap [x, y].
$$
\end{enumerate} 
Then there exists a set $M\subset \kappa$ which is not the union of fewer than $\kappa$ sets~$P_\alpha$.
\end{lemma}

\begin{proof}
We recursively define ordinals $x_\alpha,y_\alpha<\kappa$ and sets 
$M_\alpha\subset [x_\alpha,y_\alpha]$ so that 
\begin{enumerate}
\item[\textup{(i)}]
$\beta < x_\beta<y_\beta<x_\alpha<y_\alpha$ whenever $\beta<\alpha<\kappa$;
\item[\textup{(ii)}]
for any $C\subset \sup_{\beta<\alpha}x_\beta$ (in particular, for any $C\subset \alpha$), 
$$
M_\alpha\ne \bigcup_{\beta\in C}P_\beta\cap [x_\alpha, y_\alpha].
$$
\end{enumerate} 
The set $M=\bigcup_{\alpha<\kappa}M_\alpha$ is as required. Indeed, suppose that $C\subset \kappa$, $|C|<\kappa$  
and $M=\bigcup_{\beta\in C}P_\beta$. Then $C\subset \alpha$ for some $\alpha<\kappa$ (because $\kappa$ is 
regular). Clearly, $M\cap [x_\alpha, y_\alpha]=M_\alpha$, whence $M_\alpha = \bigcup_{\beta\in C}P_\beta\cap 
[x_\alpha, y_\alpha]$. This contradiction proves what we need. 
\end{proof}

It remains to prove the existence of $x$, $y$ and $M$ satisfying the conditions in Lemma~\ref{l4.3}. Let 
$\alpha<\kappa$. If $\bigcup_{\beta<\alpha} P_\beta\subset \alpha+1$, then we set $x=\alpha$, $y=\alpha+2$, and 
$M=\{\alpha+1\}$. Otherwise we set  $A = \{\beta<\alpha: P_\beta\setminus (\alpha + 1)\ne \varnothing\}$, 
$\gamma=\sup_{\beta\in A}\min(P_\beta\setminus (\alpha+1))$, $x=\alpha$, $y= \gamma+2$, and $M=\gamma+1$. For 
each $\beta<\alpha$, the intersection $C_\beta\cap [x,y]$ either is empty or contains an ordinal smaller than 
$\gamma+1$; therefore, $M$ cannot be represented as a union of such intersections. In view of  
Lemmas~\ref{l4.2} and \ref{l4.3} the group $G$ is pseudo-$\kappa$-compact. 
\end{proof}

\end{document}